\documentclass[11pt]{amsart}
\usepackage[active]{srcltx}
\usepackage{mathrsfs}
\usepackage[T1]{fontenc}

\usepackage{latexsym,enumerate}

\usepackage{amsmath,amssymb,amsthm,amsfonts,latexsym}
\usepackage{pstricks, pst-node, pst-text, pst-3d}
\usepackage[bookmarks]{hyperref}

\hypersetup{backref, colorlinks=true}

\def\adh#1{\overline{#1}}

\let\=\partial

\newtheorem {pro}{Proposition}[section]
\newtheorem {thm}[pro]{Theorem}%[section]
\newtheorem {cor}[pro]{Corollary}%[section]
\newtheorem{lem}[pro]{Lemma}

\theoremstyle{definition}
 \newtheorem {rem}[pro]{Remark}%[section]
\newtheorem {dfn}[pro]{Definition}%[section]

\newcommand{\jac}{\mbox{jac}}
\newcommand{\tet}{\tilde{\Theta}}
\newcommand{\tra}{\mathbf{tr}}

\newcommand{\R}{\mathbb{R}}

\newcommand{\N}{\mathbb{N}}
\newcommand{\Cc}{\mathscr{C}}
\newcommand{\Rr}{\mathscr{R}}

\newcommand{\cfr} {\mathfrak{c}}

\newcommand{\et}{\quad \mbox{and} \quad }

\newcommand{\hn}{\mathcal{H}}
\newcommand{\psit}{\widetilde{\psi}}

\newcommand{\St}{\mathcal{D}}

\newcommand{\mba}{ {\overline{M}}}

\newcommand{\W}{W}

\newcommand{\D}{\mathcal{D}}

\newcommand{\ep}{\varepsilon}

\newcommand{\E}{\mathcal{E}}

\newcommand{\pa}{\partial}

\newcommand{\hh}{\mathcal{V}}

\newcommand{\bou}{\mathbf{B}}
\newcommand{\sph}{\mathbf{S}}
\newcommand{\Bb}{\overline{ \mathbf{B}}}

\newcommand{\xo}{{x_0}}
\newcommand{\mc}{{\check{M}}}

\newcommand{\supp}{\mbox{\rm supp}}
\title[]{Trace operators on bounded subanalytic manifolds}



\author[A. Valette  and G. Valette]{Anna Valette  and Guillaume Valette}

\address[A. Valette]{Katedra Teorii Optymalizacji i Sterowania, Wydzia\l\ Matematyki i Informatyki Uniwersytetu Jagiello\'nskiego, ul. S. \L ojasiewicza 6, Krak\'ow, Poland}
\email{anna.valette@im.uj.edu.pl}\address[G. Valette]{Instytut Matematyki Uniwersytetu
Jagiello\'nskiego, ul. S. \L ojasiewicza 6, Krak\'ow, Poland}\email{guillaume.valette@im.uj.edu.pl}


\thanks{Research partially supported by the NCN grant  2014/13/B/ST1/00543.}


\begin{document}
\begin{abstract}
We prove that if $M\subset \R^n$ is a bounded subanalytic submanifold of $\R^n$ such that $\bou(x_0,\ep)\cap M$ is connected for every $x_0\in\mba$ and $\ep>0$ small, then, for $p\in [1,\infty)$ sufficiently large, the space $\Cc^\infty(\mba)$ is dense in the Sobolev space $\W^{1,p}(M)$. We also show that for $p$ large, if $A\subset \mba\setminus M$ is subanalytic then the restriction mapping $   \Cc^\infty(\mba)\ni u\mapsto u_{|A}\in L^p(A)$ is continuous (if $A$ is endowed with the Hausdorff measure), which makes it possible to define a trace operator, and then prove that  compactly supported functions are dense in the kernel of this operator. We finally generalize these results to the case where our assumption of connectedness at singular points of $\mba$ is dropped.
\end{abstract}
\maketitle
\begin{section}{Introduction}
In \cite{poincineq}, we proved a Poincar\'e type inequality for the functions of Sobolev spaces  of bounded subanalytic open subsets of $\R^n$. The novelty was that the boundary of the domain was not assumed to be Lipschitz regular. In the present article, we continue our study of Sobolev spaces of subanalytic domains and investigate the trace operator on  $\W^{1,p}(M)$, where $M$ is a bounded subanalytic submanifold of $\R^n$, in the case where $p$ is large.   This manifold $M$ may of course admit singularities in its closure which are not metrically conical.

%The case where $M$ is a bounded open subsets of $\R^n$ is actually one of th
The trace operator plays a crucial role in the theory of partial differential equations, as it helps to find weak formulations of the problems. This theory, which is very satisfying on domains that have Lipschitz regular boundary,  is much more challenging when singularities arise \cite{ma}. 
Our ultimate aim is therefore to develop all the material necessary to find weak formulations of basic problems of PDE on a subanalytic or semi-algebraic open subset of $\R^n$, such as for instance elliptic differential  equations with Dirichlet boundary conditions. Our method, which relies on constructions and techniques that emanate from real algebraic geometry from which many algorithms were implemented,   is highly likely to provide effective algorithms for computing approximations of solutions of partial differential equations on semi-algebraic domains.
% for example by finitely many inequalities on polynomials.

We first show that if a subanalytic bounded submanifold $M\subset \R^n$ is normal (Definition \ref{dfn_embedded}), i.e., if $M$ is connected at every point of its frontier, then, for every sufficiently large $p$ (not infinite), the set $\Cc^\infty(\mba)$ is dense  in $\W^{1,p}(M)$ (Theorem \ref{thm_trace} $(i)$). The condition of being normal is  proved to be necessary (Corollary \ref{cor_normal}). All the results that we establish are no longer true if we drop the assumption ``$p$ large enough'', and the value from which the theorem holds heavily depends on the Lipschitz geometry of the manifold near the singularities of the boundary.  
This result, which generalizes the famous theorem which is known on manifolds with Lipschitz boundary, is useful to define the trace, and we prove that, if $A$ is a subanalytic subset of $\mba\setminus M$ then the operator $\Cc^\infty(\mba)\ni u\mapsto u_{|A}\in L^p(A)$ is bounded, if $A$ is endowed with the Hausdorff measure (for $p$ large enough, Theorem \ref{thm_trace} $(ii)$). We also show (Theorem \ref{thm_trace} $(iii)$, still assuming the manifold to be normal, see also Remark \ref{rem_addendum} (\ref{item_stra})) that, when $A$ has pure dimension $k$, every  function that belongs to the kernel of this operator can be approximated by functions of $\Cc^\infty(\mba)$ that vanish in the vicinity of $A$. In particular, compactly supported functions are dense among the functions that vanish on the boundary   (Corollary \ref{cor_densite_nonnormal}). The situation is actually a bit tricky here and one has to be careful when the set $A$ is composed by several sets of different dimension (the corresponding Hausdorff measure has then to be considered on the different parts) and this is the reason why the result is stated with a stratification. 

The basic idea of our approach relies on the very precise description of the Lipschitz geometry obtained by the second author in \cite{gvpoincare, livre} (Theorem \ref{thm_local_conic_structure}), which is the natural generalization to Lipschitz geometry of the topological conic structure,  well-known to subanalytic and o-minimal geometers \cite{vdd, costeomin}. This theorem makes it possible for us to construct mollifying operators near singularities. % where it was shown that the germ of a subanalytic $M\subset \R^n$ at a point $x_0\in \mba$ can be retracted onto the point $x_0$ by a subanalytic Lipschitz deformation $r(x,s)$, $x\in M$, $s\in [0,1]$, such that $r_s$  bi-Lipschitz for every $s\in(0,1)$. 

In order to deal with the case  of non necessarily normal manifolds, we introduce and construct in the last section the  $\Cc^\infty$ normalizations, which are inspired from the  $\Cc^0$ normalizations of pseudomanifolds \cite{mcc,ih1}. This enables us to define a multi-valued trace by considering a normalization of the manifold, and very close results about density of smooth functions and trace operators can then be achieved. 
Our construction of $\Cc^\infty$ normalizations forcing us to increase the codimension of $M$ in $\R^n$, the case of a submanifold of positive codimension in $\R^n$ thus  turns out to be useful even if one is only interested in the case of an open subset of $\R^n$.     
%  \begin{rem}\label{rem_strata}
%  Given several sets $X_1,\dots,X_k$, it can be required that the respective corresponding homeomorphisms $H_1,\dots,H_k$, obtained by applying the theorem to each of the $X_i$'s, are all induced by the same homeomorphism $H:\bou(x_0,\ep)\to \bou(x_0,\ep)$ . 
%  \end{rem}
\end{section}

\begin{subsection}{Some notations}Throughout this article,  $n$, $j$, and $k$ will stand for  integers. By ``manifold'' we will always mean $\Cc^\infty$ manifold and by ``smooth mapping'' we mean $\Cc^\infty$ mapping.  The letter 
$M$ will stand for a bounded subanalytic submanifold of $\R^n$ and $m$ for its dimension.

The origin of $\R^n$ will be denoted $0_{\R^n}$. When the ambient space will be obvious from the context, we will however omit the subscript $\R^n$. We write $<,>$ for the euclidean inner product of this space  and $|.|$  for the euclidean norm.  Given $x\in  \R^n$ and $\ep>0$, we respectively denote by $\sph(x,\ep)$ and  $\bou(x,\ep)$ the sphere and the open ball of radius $\ep$ that are centered  at $x$ (for the euclidean norm), while $\adh{\bou}(x,\ep)$ will stand for the corresponding closed ball.
% We also write $\overline{B}(x,\ep)$ for  the corresponding closed ball. 
 Given a subset $A$ of $\R^n$, we denote the closure of $A$ by $\adh{A}$ and set $\delta A=\adh{A}\setminus A$. 

 We denote by $\hn^k$ the $k$-dimensional Hausdorff measure and by $L^p(A,\hn^k)$ the set of $L^p$ measurable functions on $A$ for the measure $\hn^k$. If $E\subset \R^n$ and $i\in \N\cup \{\infty\}$, we will write $\Cc^i(E)$ for the space of those functions on $E$ that extend to a $\Cc^i$ function on an open neighborhood of $E$ in $\R^n$. 
 
A mapping $h:A\to B$, $A\subset \R^n, B\subset \R^k$, is {\bf Lipschitz} if there is a constant $C$ such that $|h(x)-h(x')|\le C|x-x'|$ for all $x$ and $x'$. It is {\bf bi-Lipschitz} if it is a homeomorphism and if in addition $h$ and $h^{-1}$ are both Lipschitz.  Given two nonnegative functions $\xi$ and $\zeta$ on a set $A$ as well as a subset $B$ of $A$, we write ``$\xi\lesssim \zeta$ on $B$'' when there is a constant $C$ such that $\xi(x) \le C\zeta(x)$ for all $x\in B$.

 A submanifold of $\R^n$ will always be endowed with its canonical measure, provided by volume forms. As integrals will always be considered with respect to this measure, we will not indicate the measure when integrating on a manifold. Given a measurable function $u$ on $M$, for each $p\in [1,\infty)$ we denote by $||u||_{L^p(M)}$ the (possibly infinite) $L^p$ norm of $u$  (with respect to the canonical measure of $M$). As usual, we denote by $L^p(M)$ the set of measurable functions on $M$ for which  $||u||_{L^p(M)}$ is finite.

We then let $$\W^{1,p}(M):= \{u\in L^p(M),\; |\partial u| \in L^p(M)\}$$ denote the Sobolev space, where $\pa u$ stands for the gradient of $u$ in the sense of distributions.  

We will several times make use of the fact that a bi-Lipschitz homeomorphism relating two smooth manifolds $h:M\to M'$ identifies the respective Sobolev spaces and that we have  $\pa(u\circ h)=^{\mathbf {t}}\hspace{-2mm}Dh (\pa u\circ h)$ almost everywhere, for all $u\in W^{1,p}(M')$ (see for instance \cite[proof of Theorem III.2.13]{boyer}). It is also well-known that this space, equipped with the norm
$$||u||_{\W^{1,p}(M)}:=||u||_{L^p(M)}+| |\partial u||_{L^p(M)}$$ is a Banach space, in which
  $\Cc^\infty(M)$ is dense  for all $p\in [1,\infty)$. We will regard $\Cc^\infty(\mba)$ as a subset of $\W^{1,p}(M)$, and we will write $\Cc_0^\infty(M)$ for the set of elements of $\Cc^\infty(M)$ that are compactly supported.
 
 Given a mapping $h:M\to M$, we write $\jac\, h$  for the absolute value of the determinant of $Dh$, the derivative of $h$ (defined at the points where $h$ is differentiable). 
 We write $p'$ for the H\"older conjugate of $p$, i.e., $p'=\frac{p}{p-1}$.

\end{subsection}

\begin{section}{Subanalytic sets}
 We refer the reader to  \cite{bm, ds, l, livre} for all the basic facts about subanalytic geometry.  Actually, \cite{livre} also gives a detailed presentation of the Lipschitz properties of these sets, so that the reader can find there all the needed facts to understand the result achieved in the present article.

\begin{dfn}\label{dfn_semianalytic}
A subset $E\subset \R^n$ is called {\bf semi-analytic} if it is {\it locally}
defined by finitely many real analytic equalities and inequalities. Namely, for each $a \in   \R^n$, there is
a neighborhood $U$ of $a$ in $\R^n$, and real analytic  functions $f_{ij}, g_{ij}$ on $U$, where $i = 1, \dots, r, j = 1, \dots , s_i$, such that
\begin{equation}\label{eq_definition_semi}
E \cap   U = \bigcup _{i=1}^r\bigcap _{j=1} ^{s_i} \{x \in U : g_{ij}(x) > 0 \mbox{ and } f_{ij}(x) = 0\}.
\end{equation}

The flaw of the  semi-analytic category is that  it is not preserved by analytic morphisms, even when they are proper. To overcome this problem, we prefer working with the  subanalytic sets, which are defined as the projections of the semi-analytic sets.

A subset $E\subset \R^n$  is  {\bf  subanalytic} if 
 each point $x\in\R^n$ has a neighborhood $U$ such that $U\cap E$ is the image under the canonical projection $\pi:\R^n\times\R^k\to\R^n$ of some relatively compact semi-analytic subset of $\R^n\times\R^k$ (where $k$ depends on $x$).
   
   A subset $Z$ of $\R^n$ is  {\bf globally subanalytic} if $\hh_n(Z)$ is a subanalytic subset of $\R^n$, where $\hh_n : \R^n  \to (-1,1) ^n$ is the homeomorphism defined by $$\hh_n(x_1, \dots, x_n) :=  (\frac{x_1}{\sqrt{1+|x|^2}},\dots, \frac{x_n}{\sqrt{1+|x|^2}} ).$$

   We say that {\bf a mapping $f:A \to B$ is  subanalytic} (resp. globally subanalytic), $A \subset \R^n$, $B\subset \R^m$ subanalytic (resp. globally subanalytic), if its graph is a  subanalytic  (resp. globally subanalytic) subset of $\R^{n+m}$. In the case $B=\R$, we say that  $f$ is a (resp. globally) {\bf  subanalytic function}. For simplicity globally subanalytic sets and mappings will be referred as {\bf definable} sets and mappings (this terminology is often used by o-minimal geometers \cite{vdd,costeomin}).

   \end{dfn}

The globally subanalytic category is very well adapted to our purpose.  It is stable under intersection, union, complement, and projection. It thus constitutes an o-minimal structure \cite{vdd, costeomin}, and consequently admits cell decompositions and stratifications, from which it comes down that definable sets enjoy a large number of finiteness properties (see \cite{costeomin,livre} for more).

Clearly,  a bounded subset of $\R^n$ is subanalytic if and only if it is globally subanalytic. This is the reason why we generally do not mention ``globally'' when working with bounded sets. Note that a subanalytic function $f$ on a bounded set $A$ may be unbounded and therefore not globally subanalytic.

Our results about $L^p$ functions will be valid {\it for sufficiently large $p$}, which means that there will be $p_0\in \R$ such that the claimed fact will be true for all $p\in (p_0,\infty)$.  This number $p_0$ will be provided by \L ojasiewicz's inequality, which is one of the main tools of subanalytic geometry. It originates in the fundamental work of S. \L ojasiewicz \cite{lojdiv}, who established this inequality in order to answer a problem  about distribution theory.  We shall make use of the following version (see the proof of (\ref{ln})):

%  An important tool in subanalytic geometry is \L ojasiewicz's inequality, here we recall it in a slightly modified form.
 \begin{pro}\label{pro_lojasiewicz_inequality}(\L ojasiewicz's inequality)
Let $f$ and $g$ be two globally subanalytic functions on a  globally subanalytic set $A$ with $\sup\limits_{x\in A} |f(x)|<\infty$. Assume that
 $\lim\limits_{t \to 0} f(\gamma(t))=0$ for every
globally subanalytic arc $\gamma:(0,\ep) \to A$ satisfying $\lim\limits_{t \to 0} g(\gamma(t))=0$.

Then there exist $\nu \in \N$ and $C \in \R$ such that for any $x \in A$:
$$|f(x)|^\nu \leq C|g(x)|.$$
\end{pro}

See  for instance \cite{li} for a proof. 
 The main ingredient of our approach is the following result achieved by the second author in \cite{gvpoincare} (see also \cite{livre} for a complete expository of the theory) which describes the conic structure of subanalytic sets from the metric point of view. In the theorem below, $x_0* (\sph(x_0,\ep)\cap X)$ stands for the cone over $\sph(x_0,\ep)\cap X$ with vertex at $x_0$.

\begin{thm}[Lipschitz Conic Structure]\label{thm_local_conic_structure}
  Let  $X\subset \R^n$ be subanalytic and $x_0\in X $. %, and set for simplicity $X^\ep = B(x_0,\ep)\cap X$.
For $\ep>0$ small enough, there exists a Lipschitz subanalytic homeomorphism
$$H: x_0* (\sph(x_0,\ep)\cap X)\to  \Bb(x_0,\ep) \cap X,$$  
  satisfying $H_{| \sph(x_0,\ep)\cap X}=Id$, preserving the distance to $x_0$, and having the following metric properties:
\begin{enumerate}[(i)] 
%\item\label{item_dist}    $h$ preserves the distance to $x_0$. 
 \item\label{item_H_bi}     The natural retraction by deformation onto $x_0$ $$r:[0,1]\times  \Bb(x_0,\ep)\cap X \to \Bb(x_0,\ep)\cap X,$$ defined by $$r(s,x):=H(sH^{-1}(x)+(1-s)x_0),$$ is Lipschitz.   %\item  Moreover, for every $s\in (0,1]$, the retraction $r_s$, defined by $x\mapsto r_s(x):=r(s,x)$ is  Lipschitz.\item\label{item_r_cs_lip}  
 Indeed, there is a constant $C$ such that  for every fixed $s\in [0,1]$, the mapping $r_s$ defined by $x\mapsto r_s(x):=r(s,x)$, is $Cs$-Lipschitz.
 \item \label{item_r_bi}  For each $\delta>0$,
 the restriction of $H^{-1}$ to $\{x\in X:\delta \le |x-x_0|\le \ep\}$ is Lipschitz and, for each $s\in (0,1]$, the map  $r_s^{-1}:\Bb(x_0,s\ep) \cap X\to \Bb(x_0,\ep) \cap X$ is Lipschitz. %the complement of $B(x_0,\eta)$ in  $x_0* (S(x_0,\ep)\cap X)$ is bi-Lipschitz. 
\end{enumerate}
\end{thm}

\begin{subsection}{Stratifications}  %If a set is not smooth, it is natural to split it into smooth objects.  
% In this section we gather basic facts about stratifications and we construct a stratification which is adapted to our purpose.

\begin{dfn}\label{dfn_stratifications}
 A {\bf 
stratification} of a subset of $ \R^n$ is a finite partition of it into
definable $\Cc^\infty$ submanifolds of $\R^n$, called {\bf strata}. A stratification is {\bf compatible} with a set if this set is the union of some strata. A {\bf refinement} of a stratification $\Sigma$ is a stratification $\Sigma'$ compatible with the strata of $\Sigma$. 

% If $S\subset \R^n$ is a manifold, a {\bf local retraction} of $S$ is a  $\pi:U\to S$, where 

  A  stratification $\Sigma$ of a set $X$ is {\bf locally bi-Lipschitz trivial} if for every   $S\in \Sigma$ and $x_0\in S$,  there are a neighborhood $W$ of $x_0$ in $S$ and a $\Cc^\infty$ definable retraction $\pi:U\to W$, with $U$ neighborhood of $x_0$ in $X$, as well as a  bi-Lipschitz homeomorphism  $$\Lambda:U\to (\pi^{-1}(x_0)\cap X) \times W, $$ satisfying:
  \begin{enumerate}[(i)]
   \item  $\pi(\Lambda^{-1}(x,y))= y$, for all $(x,y)\in (\pi^{-1}(x_0)\cap X)\times  W$.
 \item   $\Sigma_{x_0}:=\{  \pi^{-1}(x_0)\cap Y:Y\in \Sigma\} $ is a stratification of $ \pi^{-1}(x_0)\cap X$,  and $\Lambda(Y\cap U)=(\pi^{-1}(x_0)\cap Y)\times W$, for all $Y\in \Sigma$.%  sending $Y\cap U$ onto , for each $Y\in \Sigma$. 
  \end{enumerate}
   % It is said to be {\bf  locally bi-Lipschitz trivial} if so it is at each $x_0\in X$.
\end{dfn}

 \begin{thm}\label{pro_existence_stratifications}
  Every stratification $\Sigma$ can be refined into a stratification which is locally bi-Lipschitz trivial. 
 \end{thm}

Lipschitz stratifications, in the sense of T. Mostowski \cite{m}, are locally bi-Lipschitz trivial. Existence of Lipschitz stratifications for subanalytic sets was proved by A. Parusi\'nski \cite{stratlip}. The bi-Lipschitz version of Hardt's theorem published in \cite{vlt} enables to construct locally definably bi-Lipschitz trivial stratifications,  in the sense that it is possible to construct locally  bi-Lipschitz trivial stratifications for which the local trivializations are in addition definable (see \cite[Proposition 4.5.8]{livre} for more details). 

   \begin{rem}\label{rem_pi_S} \begin{enumerate}
                                \item   All the results of this article are actually still valid with no modification in the proofs in the larger framework provided by polynomially bounded o-minimal structures admitting $\Cc^\infty$ cell decompositions. Existence of Lipschitz stratifications on such a structure was established in \cite{lipsomin}. With minor modifications in the proofs, one could actually also prove that the trace is an $L^p$ bounded operator without assuming existence of $\Cc^\infty$ cell decompositions (see Remark \ref{rem_normalisations_C_i}).
\item A stratification is said to be Whitney $(a)$ regular if for every sequence of points $y_i$ in  a stratum $Y$ tending to some point $z$ in another stratum $Z$  in such a way that $T_{y_i}Y$ converges to a vector space $\tau$ (in the Grassmannian), we have $T_z Z\subset \tau$.
 It is easy to see that if $\Sigma$ is  a Whitney $(a)$ regular stratification of $\mba$ of which $M$ is a stratum, $S$ is a stratum included in $\delta M$, and $\pi:U\to S$ is a smooth retraction, then 
  $\pi^{-1}(x)\cap M$ is a smooth manifold for every $x\in S$, after shrinking $U$ to a smaller neighborhood of $S$ if necessary. 
In the proofs, we will assume that our stratifications refine such a stratification and therefore that $\pi^{-1}(x)\cap M$ is a manifold  (the existence of such a stratification $\Sigma$ is well-known \cite{livre}).
                               \end{enumerate}
 \end{rem}

\end{subsection}

% \begin{rem}

\end{section}

% \begin{section}{Normal manifolds}

% \end{section}

\begin{section}{The case of normal manifolds}\label{sect_case_normal} Let us assume that $0_{\R^n}\in\adh{M}$ and 
set for $\eta>0$: $$M^{\eta}:=\bou(0_{\R^n},\eta)\cap M\;\; \et\;\; N^{\eta}:=\sph(0_{\R^n},\eta)\cap M.$$ %where $\ep$ is given by Theorem \ref{thm_local_conic_structure} and  

  Apply Theorem \ref{thm_local_conic_structure} with $X=M\cup\{0_{\R^n}\}$ and $x_0=0_{\R^n}$. Fix $\ep>0$ sufficiently small for the statement of the theorem to hold and for $N^\ep$ to be a smooth manifold, and  let $r$ denote the mapping provided by this theorem. Observe  that since $r$ is subanalytic, it is smooth almost everywhere.

The definition of the trace will appear in section \ref{sect_trace} (Theorem \ref{thm_trace}). It requires some preliminary local computations that are presented in sections \ref{sect_theta} and \ref{sect_r}. 
We first establish in section \ref{sect_key} some estimates that rely on the information we have about the derivative of $r$. They shed light on the role that \L ojasiewicz's inequality will play.

% \noindent{\bf A few inequalities.}
\begin{subsection}{Some key facts.}\label{sect_key}  There is a positive constant $C$ such that: 
 \begin{enumerate}
\item For all $s\in (0,1)$ we have for almost all $x\in M^\ep$:  \begin{equation}\label{eq_der_r_s}
       \left|\frac{\pa r}{\pa s}(s,x)\right|\le C|x|.
      \end{equation}
       \item For each $v\in L^{p}(M^\ep)$,  $p\in [1,\infty)$, we have for all $\eta\in (0,\ep]$: \begin{equation}\label{eq_coarea_sph}
  \frac{1}{C}  \left(\int_0 ^\eta ||v||_{L^p (N^\zeta)}^p d\zeta \right)^{1/p}  \le    ||v||_{L^p (M^\eta)} \leq C \left(\int_0 ^\eta ||v||_{L^p (N^\zeta)}^p d\zeta \right)^{1/p}.
       \end{equation}
       \item There exists $\nu\in \N$ such that for each $v\in L^{p}(M^\ep)$,  $p\in [1,\infty)$,  $\eta\in(0,\ep)$, and $s\in (0,1)$:
\begin{equation}\label{ln}
||v\circ r_s||_{L^p(N^\eta)}\le Cs^{-\nu/p}||v||_{L^p(N^{s\eta})}  , 
\end{equation}
% where $\nu $ is given by (\ref{eq_jacr_s}).
 \end{enumerate}
 
 %\end{lem}
% \end{rem}
\begin{proof} By definition of $r$, $\frac{\pa r}{\pa s}(s,x)=D_{sH^{-1}(x)} H(H^{-1}(x))$, where $H$ is provided by Theorem \ref{thm_local_conic_structure} (recall that $x_0=0$). As a  matter of fact, (\ref{eq_der_r_s}) immediately comes down from the fact that $H^{-1}$ preserves the distance to the origin and $H$ has bounded derivative.

Let $\rho: M\to \R$ be defined by $\rho(x):=|x|$. To prove the second estimate, note that we have (see for instance \cite{krantz}, Theorem $5.3.9$):
$$        \int_0 ^\eta \int_{N^\zeta} |v(x)|^p dx d\zeta
 = \int_{M^\eta} |v(x)|^p \left|\pa \rho(x)\right|dx. $$
Hence, it suffices to check that $\left| \pa \rho\right|$ is bounded away from $0$. Let $\tilde{\rho}:=\rho \circ H$. Because $H$ preserves the distance to the origin and $H^{-1}(M^\ep)$ is a cone, we have $\pa \tilde{\rho}(y)=\frac{y}{|y|}$, for all $y$,  which means that $\pa \rho(x)=\frac{^{\bf t}D_xH^{-1} (H^{-1}(x))}{|H^{-1}(x)|}$, for all $x$.  But, as $H$ is  Lipschitz, $\inf_{|z|=1} |{^{\bf t}D_x H^{-1}(z)}|$ must be bounded away from zero independently of $x$, yielding the required fact.

To show the last inequality, observe that, since $r_s$ is bi-Lipschitz for every $s>0$, $\jac\, r_s$ can only tend to zero if $s$ is itself going to zero. Hence, by \L ojasiewicz's inequality (see Proposition \ref{pro_lojasiewicz_inequality}),
there are an integer $\nu$ and a constant $C$ such that  for all $s\in (0,1)$ we have for almost all $x\in M^\ep$:
\begin{equation}\label{eq_jacr_s}
 \jac \,r_s(x) \ge \frac{s^{\nu}}{C}.  
\end{equation}
To prove (\ref{ln}), it thus suffices to write
$$||v\circ r_s||_{L^p (N^\eta)}^p= \left(\int_{N^\eta}|v\circ r_s|^p\right)^{1/p}\overset{(\ref{eq_jacr_s})}{\lesssim} \left(\int_{N^\eta} 
| v(r_s(x))|^p \frac{\jac \, r_s (x)}{s^\nu}\ dx \right)^{1/p},
$$and to remark that $r_s(N^\eta)=N^{s\eta}$.
\end{proof}
\begin{rem}
Fact (3) is still valid if we replace $N$ with $M$, i.e., there is $\nu\in \N$ such that: \begin{equation}\label{lnm}
||v\circ r_s||_{L^p(M^\eta)}\lesssim s^{-\nu/p}||v||_{L^p(M^{s\eta})}  , 
\end{equation} for $v\in L^p(M^\ep)$,  $p\in [1,\infty)$, $s\in(0,1)$, and $\eta<\ep$. This can be proved directly (like above for $N^\eta$) or deduced from (\ref{ln}) and (\ref{eq_coarea_sph}).
\end{rem}
\end{subsection}

\begin{subsection}{The operator $\Theta^M$.}\label{sect_theta}
Let us set for $u\in \W^{1,p}(M^\ep)$ and $x\in M^\ep$:   
\begin{equation}\label{theta}\Theta^M u(x):= \int_0^1\frac{\partial (u\circ r)}{\partial s}(s,x)\ ds= \int_0^1<\partial u(r_s(x)),\frac{\partial r}{\partial s}(s,x)>\, ds.\end{equation}
% Actually, since $r_0\equiv 0_{\R^n}$ which does not necessarily belong to $M$,  
This definition makes sense (for almost every $x$) as soon as the considered function is sufficiently integrable, as established by the lemma below.

\begin{lem}\label{theta_ver}
For $p$ sufficiently large, the function $[0,1] \ni s \mapsto ||\frac{\partial (u\circ r_s)}{\partial s}||_{L^p(M^\ep)}$ belongs to $L^1([0,1])$ for all $u \in \W^{1,p}(M^\ep)$, so that $\Theta^M u$ is well-defined. Moreover, for $u\in \W^{1,p}(M^\ep)$ and  $\eta<\ep$, we then have:
\begin{equation}\label{theta_cont}||\Theta^M u||_{L^p(N^\eta)}\lesssim \eta^{1-\frac{1}{p}}\,||u||_{\W^{1,p}(M^\eta)}.\end{equation}
\end{lem}

\begin{proof}
For  $u \in \W^{1,p}(M^\ep)$, we have:
\begin{eqnarray*}
%=\left(\int_{M^\ep}\left(\int_0^1\frac{\partial }{\partial s}u\circ r(s,x)\ ds\right)^p dx\right)^{1/p}&\le \\
\int_0^1||\frac{\partial (u\circ r_s)}{\partial s}||_{L^p(M^\ep)} ds&=& \int_0^1\left(\int_{M^\ep}\left |\frac{\partial (u\circ r)}{\partial s}(s,x)\right |^p\ dx \right)^{1/p} ds  \\
&\lesssim&\int_0^1\left(\int_{M^\ep} 
|\partial u(r_s(x))|^p\ dx \right)^{1/p} ds \quad\mbox{(since $r$ is Lipschitz)} \\
% &\overset{(\ref{eq_der_r_s})}{\le}& A\int_0^1\left(\int_{M^\ep} 
% |\partial u(r_s(x))|^ps^p\ dx \right)^{1/p} ds \\
%&\overset{(\ref{eq_jacr_s})}{\lesssim}& \int_0^1\left(\int_{M^\ep} 
%|\partial u(r_s(x))|^p \frac{\jac \, r_s (x)}{s^\nu}\ dx \right)^{1/p} ds \\
&\overset{(\ref{lnm})}{\lesssim}& \int_0^1 s^\frac{-\nu}{p}||\pa u||_{L^p (M^{s\ep})}  ds \\%\quad \mbox{  (see (\ref{eq_jacr_s}) for $\nu$)} \\
&\le& ||\pa u||^p_{L^p (M^{\ep})}\int_0^1|s|^{-\nu /p}ds ,  
% &\le& \left(\int_0^1|s|^{-\nu p'/p}ds\right)^{1/p'}\left (\int_0^1||\pa u||^p_{L^p (M^{s\ep})} ds\right)^{1/p},  \\
% &\lesssim& ||\pa u||^p_{L^p (M^{s\ep})}  
%  \\
%&\le&\int_0^1\left(\int_{M^\ep} 
%|\partial u(r_s(x))|^p\left|\frac{\partial r}{\partial s}(s,x)\right |^p\ dx \right)^{1/p} ds \\
%&\le& A\int_0^1\left(\int_{M^\ep} 
%|\partial u(r_s(x))|^p\left|\frac{\jac (r_s(x))}{s^\nu}\right |\ dx \right)^{1/p} ds \\
% &=& A\int_0^1|s|^{-\nu/p}\left(\int_{M^\ep} 
% |\partial u(r_s(x))|^p\jac (r_s(x))\ dx \right)^{1/p} ds\\
% &=& A\int_0^1|s|^{-\nu/p}||\partial u||_{L^p(r_s(M^\ep))} ds\\
% &=& A ||\partial u||_{L^p(M^\ep)}\int_0^1|s|^{-\nu/p} ds,$||\Theta^M u||_{L^p(M^\ep)}$
\end{eqnarray*}
which is  finite for $p>\nu$. That $\Theta^M u$ is well-defined now follows from Minkowski's integral inequality. This proves the first part of the Lemma.
For the second part, applying again  Minkowski's integral inequality, we can write:
\begin{eqnarray*}
||\Theta^M u||_{L^p(N^\eta)}%\int_{N^\eta}\left|\int_0^1\frac{\partial }{\partial s}u\circ r(s,x)\ ds \right|^p dx&\le \\
%&\le& \int_0^1\left(\int_{N^\eta}\left |\frac{\partial }{\partial s}u\circ r(s,x)\right |^p\ dx \right)^{1/p} ds  \\
%By Cauchy-Schwartz inequality note that we can assume that $u$ is smooth on $M$.    A
&\le&\int_0^1\left(\int_{N^\eta} 
|\partial u(r_s(x))|^p\left|\frac{\partial r}{\partial s}(s,x)\right |^p\ dx \right)^{1/p} ds \\
&\overset{(\ref{eq_der_r_s})}{\lesssim}& \eta\int_0^1\left(\int_{N^\eta} 
|\partial u(r_s(x))|^p\ dx \right)^{1/p} ds \\
&\overset{(\ref{ln})}{\lesssim}&\eta\int_0^1|s|^{-\nu/p}  
||\partial u||_{L^p(N^{s\eta})}ds\\
&\le& \eta\left(\int_0^1|s|^{-\nu p'/p}ds\right)^{1/p'}\left(\int_0^1 ||\partial u||_{L^p(N^{s\eta})}^pds\right)^{1/p}\\
&\lesssim& \eta\left(\int_0^1 ||\partial u||_{L^p(N^{s\eta})}^pds\right)^{1/p}\;\; \mbox{\rm (for \; $p>\nu p'$)}\\
&=& \eta^{1-1/p} \left(\int_0^\eta 
||\partial u||^p_{L^p(N^t)}dt\right)^{1/p}\quad (\mbox{setting $t:=s\eta$})\\
&\overset{(\ref{eq_coarea_sph})}{\lesssim}& \eta^{1-1/p} ||\partial u||_{L^p(M^\eta)}.
\end{eqnarray*}
\end{proof}
\begin{rem}
Observe that (\ref{theta_cont}) and (\ref{eq_coarea_sph}) yield that for $p$ large enough, and  $u$ and $\eta$ as in (\ref{theta_cont}), we have:
\begin{equation}\label{theta_contm}||\Theta^M u||_{L^p(M^\eta)}\lesssim\eta ||u||_{\W^{1,p}(M^\eta)}.\end{equation}
\end{rem}
\end{subsection}

\begin{subsection}{The operator $\Rr^M$.}\label{sect_r}  Set  now for $u \in\W^{1,p}(M^\ep)$ (and $p$ sufficiently large for $\Theta^M$ to be defined):
\begin{equation}\label{R}\Rr^M u:=u-\Theta^M u.\end{equation}
%We now introduce another operator which will turn out to be the complement projection.

\begin{lem}\label{lem_cont}For $p$ sufficiently large, $\Rr^M u$ is constant on every connected component of $M^\ep$, for all $u\in \W^{1,p}(M^\ep)$. Moreover, $\Theta^M$ and $\Rr^M$ are then continuous projections and, in the case where $u$ extends to a continuous function on $\adh{M^\ep} $, we have $\Rr^M u\equiv u(0)$.   
\end{lem}
\begin{proof} 
Take $p$ sufficiently large for Lemma \ref{theta_ver} to hold, and, for $u\in \W^{1,p}(M^\ep)\cap \Cc^\infty(M^\ep)$, $t\in (0,1)$, as well as $x\in M^\ep$, set  \begin{equation*}\label{theta}\Theta^M_t u(x):= \int_t^1\frac{\partial \left(u\circ r\right)}{\partial s}(s,x)\ ds=u(x)-u(r_t(x)).%= \int_0^1<\partial u(r_s(x)),\frac{\partial r}{\partial s}(s,x)>\ ds
\end{equation*}By Minkowski's integral inequality, we have:
\begin{equation}\label{eq_theta_t}
 ||\Theta^M u -\Theta^M_t u||_{L^p(M^\ep)} \le \int_0 ^t ||\frac{\pa (u\circ r_s)}{\pa s} ||_{L^p(M^\ep)}ds,
\end{equation}
which, by Lemma \ref{theta_ver}, must tend to $0$ as $t$ goes to $0$.
Observe also that since $u-\Theta^M_t u=u\circ r_t$, we have for $t\in (0,1)$:
\begin{equation}\label{eq_pa_u_theta}|\pa(u-\Theta ^M_t u)|=|^{\bf t}Dr_t \left(\pa  u\circ r_t\right)|\le C t |\pa u\circ r_t|.\end{equation}                                                                                                                                
By (\ref{lnm}), the $L^p$-norm of the right-hand-side goes to $0$ as $t$ goes to $0$, for all $p$ large. 
But since (\ref{eq_theta_t}) yields that $(u-\Theta^M_t u)$ converges to $(u-\Theta^M u)$ in the $L^p$-norm, we now can write: $$\partial (\Rr^Mu)=\partial (u-\Theta^Mu)=\lim_{t\to 0}\pa(u-\Theta ^M_t u)  \overset{(\ref{eq_pa_u_theta})}{=}0,$$ where the limit is taken with respect to the $L^p$ norm. 
In particular, $\Rr^Mu$ and $\Theta^Mu$ belong to $\W^{1,p}(M^\ep)$ (for $p$ large).
% By (\ref{theta_cont}) (and (\ref{eq_coarea_sph})), we must have:  
% $$||\Theta^Mu||_{L^p(M^\ep)}\le C||u||_{\W^{1,p}(M^\ep)}$$ for some constant $C>0$.
It also means that $\partial \Theta^Mu=\partial u$, and, thanks to (\ref{theta_contm}), we derive that $\Theta^M$ is a continuous operator on $\W^{1,p}(M^\ep)$. Remark then that $\Theta^M \Rr^M u=0$, so that, applying $\Theta^M$ to (\ref{R}), we see that $\Theta^M$ is a projection.
 Finally, in the case where $u$ extends continuously on $\adh{M^\ep}$, thanks to the definition of $\Theta^M_t$ and (\ref{eq_theta_t}), we see that $\Theta^M u=u-u(0)$. 
\end{proof}
\end{subsection}
  
\begin{subsection}{The trace on a normal manifold.}\label{sect_trace} %In the case where $M$ is normal, the
%  situation is more natural and the theorem will be analogous as in the case of manifolds with Lipschitz boundary.
\begin{dfn}\label{dfn_embedded}
We say that $M$ is {\bf connected at $x\in \delta M$}\index{connected at $x$} if  $\bou(x,\ep)\cap M$ is connected for all  $\ep>0$ small enough.
We say that $M$ is {\bf normal} if it is connected at each $x\in \delta M$.
\end{dfn}

Our constructions will be inductive on the dimension of the strata that meet the support of the considered function, which requires to regard the support as a subset of $\adh{M}$. We start by introducing the necessary notations on this issue.

Let $U$ be an open subset of $M$ and let $V$ be an open subset of $\mba$ satisfying $V\cap M=U$. For a distribution $u$ on $U$, we define
$\supp_V u$ ({\bf the support of $u$ in $V$}) as
 the closure in $V$ of the set $\supp \, u$.  When $\supp_V u$ is compact, we will say that $u$ has {\bf compact support in $V$}. 
 
 If $ B$ is an open subset of $\mba$ containing $M$, regarding the elements of $\Cc^\infty(\mba)$ as functions on $M$, we then can set $$\Cc^\infty_{B}(\mba):=\{u\in \Cc^\infty(\overline{M}): \supp_B u \;\mbox{ is compact} \}.$$
  Observe that, as $M$ is bounded, given an $L^p$ function $u:M\to \R$,  $\supp_\mba u\cap \delta M=\emptyset$ if and only if $u$ is compactly supported (in $M$).

\begin{thm}\label{thm_trace}
Assume that $M$ is normal and let $A$ be a subanalytic subset of $\delta M$. For all $p\in [1,\infty)$ sufficiently large, we have:
\begin{enumerate}[(i)]
\item
$\Cc^\infty(\adh{M})$ is dense in $\W^{1,p}(M)$.
\item The linear operator
\begin{equation*}\label{trace}
\Cc^\infty(\adh{M})\ni \varphi \mapsto \varphi_{|A}\in L^p(A,\hn^k), \qquad k:=\dim A,
\end{equation*}
is bounded for $||\cdot ||_{\W^{1,p}(M)}$ and thus extends to a mapping $\tra_A:\W^{1,p}(M)\to L^p(A,\hn^k)$.
\item If  $\mathcal{S}$ is a stratification of $A$, then $\Cc^\infty_{\adh{M}\setminus\adh{A}}(\mba)$ is a dense subspace of
$\bigcap\limits_{Y\in\mathcal{S}}\ker \tra_Y$.

\end{enumerate}

\end{thm}

\begin{proof} 
\noindent {\bf Proof of $(i)$.} % (with $M$ as top stratum)
Fix a locally bi-Lipschitz trivial stratification $\Sigma$ of $\adh{M}$ compatible with $\delta M$. 
 For $u\in \W^{1,p}(V\cap M)$, with $V$ an open subset of $\adh{M}$,  we define \begin{equation}\label{eq_kappa_u}\kappa_u: =\min\{\dim S\,:S\in\Sigma, \, S\subset \delta M,\, S\cap \supp_{V} u\neq\emptyset\}.\end{equation}
 
 The idea is to argue by decreasing induction on $\kappa_u$. Namely, we prove: % be a subanalytic open subset $V$ of $\adh{M}$
 
 \noindent{$(\mathbf{A}_k).$} Let $V$ be as above and $V'$ be an open set satisfying $\adh{V'}\subset V$. Given $u\in  \W^{1,p}(V\cap M)$ satisfying $\kappa_u\ge k$, there is a family of functions $v_\mu\in \W^{1,p}(V\cap M)\cap \Cc^\infty(V)$, $\mu>0$, such that $||u-v_\mu||_{\W^{1,p}(V'\cap M)}$ tends to zero as $\mu$ goes to zero.
 
%  we can approximate any function $u\in \W^{1,p}(M)$ by smooth functions, we will  
  It is well-known that functions that are compactly supported in $M$ can be approximated by smooth compactly supported functions, which already means that $(\mathbf{A}_m)$ holds true.
We thus fix $k< m$, assume that $(\mathbf{A}_j)$ holds for all $j>k$, and take a function $u\in  \W^{1,p}(V\cap M)$, $V\subset \adh{M}$ open, with $\kappa_u\ge k$. %the result is true for every function $v\in \W^{1,p}(M)$ with $\kappa_v>k$,

Every point $x_0\in \overline{M}$ admits a neighborhood $U_{x_0}$ in $\overline{M}$ on which we have a bi-Lipschitz trivialization of $\Sigma$.  As we can use a partition of unity, we may assume that $\supp_{V} u$ fits in $U_{x_0}$, for some $x_0\in \adh{M}$.

Let $S$ denote the stratum of $\Sigma$ that contains the point $x_0$.   Observe that if $\dim S< \kappa_u$ then the function $u$ is zero near $S$, so that in such a situation, choosing $U_{x_0}$ smaller if necessary, we can require the function $u$ to be zero on $U_{x_0}$.  Moreover, in the case where $\dim S>k$,  the result follows by our decreasing induction on $\kappa_u$ (since $\supp_{V}u\subset U_{x_0}$ which cannot meet any stratum of dimension smaller than $\dim S$), which means that we can assume that $\dim S\le k$. As $k\le \kappa_u$, we will thus assume that $\dim S=\kappa_u =k$ and, without loss of generality, we will do the proof for $x_0=0_{\R^n}\in \delta M$.
% Let $\kappa_u=k$, and assume that the result holds for any $v$ such that $m\ge \kappa_v>k$.
% Let $S$ be a $k$-dimensional stratum of $\delta M$ such that $S\cap \supp_{\adh{M}}u\neq\emptyset$. 

We start with the easier case $k=0$, i.e.,  we assume for the moment that $S=\{0_{\R^n}\}$. We will also suppose that $\supp_{V} u\subset \bou (0_{\R^n},\ep)$, with $\ep$  sufficiently small for the operators $\Theta^M$ and $\Rr^M$ to be defined on $M^\ep$ (refining our partition of unity if necessary).% We thus have 
% In this case, there is $\ep>0$ such that we have on $M^\ep$:
% \begin{equation}\label{eq_urt_proof}u=\Theta^Mu +\Rr^M u.\end{equation}
                       
As $M$ is connected at $0_{\R^n}$, $\Rr^Mu $ is constant on $M^\ep$ (see Lemma \ref{lem_cont}), which implies that it is a smooth function. It thus suffices to approximate $\Theta^M u$. %Observe that by (\ref{theta_contm}), we have:
% \begin{equation}\label{theta_proof} ||\Theta^M u||_{L^p (M^{\eta})} \lesssim \eta ||u||_{\W^{1,p}(M^\eta)} .\end{equation}

Let $\psi:\R \to [0,1]$  be a $\Cc^\infty$ function  
such that $\psi\equiv 0$ on $(-\infty,\frac{1}{2})$ and $\psi\equiv 1$ near $[1,+\infty)$, and set for $\mu$ positive and $x\in M^\ep$, $\psi_\mu(x):=\psi\left(\frac{|x|^2}{\mu^2}\right)$.  Note that, as $\psi_\mu$ is bounded and tends to $1$, the function  $\psi_\mu\cdot \pa u$  tends to $\pa u$ in the $L^p$-norm as $\mu$ goes to $0$.  As a matter of fact, if we write
\begin{equation}\label{eq_v_mu_cv}|| \pa (\psi_\mu \Theta^M u)-\pa u||_{L^p(M^\ep)}\le || \Theta^M u\cdot \pa \psi_\mu ||_{L^p(M^\ep)}+||( \psi_\mu-1) \pa u||_{L^p(M^\ep)},\end{equation} 
by (\ref{theta_contm}), an easy computation shows that $v_\mu:=\psi_\mu\cdot \Theta^M u$  tends to $\Theta^M u$ in $\W^{1,p}(M^\ep)$ for all $p$ large enough (since $\sup_{x\in M} |\pa \psi_\mu(x)|\lesssim 1/\mu$ and $\supp_{M} \pa \psi_\mu\subset M^{\mu}$). As $ \kappa_{v_\mu} >0$, this completes the induction step in the case $k=0$. 

We now deal with the case $k>0$ that we shall address with a similar argument. The situation is however more complicated and we shall need to apply a mollifier along the stratum $S$.

  By definition of bi-Lipschitz triviality, there is a bi-Lipschitz homeomorphism $\Lambda: U_0 \to (\pi^{-1}(0_{\R^n})\cap \adh{M})\times W_0$ (here $U_0=U_{0_{\R^n}}$), where $\pi:U_0 \to S$ is a smooth retraction and $W_0$ is a neighborhood of the origin in $S$ that, up to a coordinate system of $S$, we can identify with $\bou(0_{\R^k},\alpha)$, $\alpha>0$. Notice that the sets $$F:=\pi^{-1}(0_{\R^n}) \cap M \et U':=F\times \bou(0_{\R^k},\alpha) $$  are smooth manifolds (see Remark \ref{rem_pi_S} (2)).
 %  As we can argue up to a subanalytic bi-Lipschitz homeormophism (see Proposition \ref{pro_pullback_weakly differentiable forms}), we can identify $U_0$ with $F\times \bou^k(0,\alpha)$.

% As $\Sigma$ is a bi-Lipschitz local stratification, every point $0\in S$ has a neighbourhood $U_0$ Taking $U_0$ smaller if necessary, we can assume that it does not meet any other stratum of dimension $k$ than $S$.
% such thatSince $\kappa_u=\dim S$, $\supp_{\adh{M}}u$ is a compact subset of $S$, so that,  up to  a partition of unity, we can assume that $\supp\ u\subset U_0$. 

For $v\in L^p(U')$, we define two families of functions on $F$ and $\bou(0_{\R^k},\alpha)$ respectively, parameterized by $\bou(0_{\R^k},\alpha)$ and $F$  respectively,  by setting for almost every $(x,y)\in U'$ $$v_x(y)=v^y(x):=v(x,y).$$ Next, we define a mollifying operator $\Phi_\sigma $, by setting for $\sigma>0$ small and  $v\in L^p(U')$: 
\begin{equation}\label{eq_Phi_sigma_def}
\Phi_\sigma v(x,y)=v_x\ast \varphi_\sigma(y)=\int_{\bou(0_{\R^k},1)} v_x(y-z)\varphi_\sigma(z)\ dz,
\end{equation}
where $\varphi_\sigma(z)=\frac{1}{\sigma^k}\varphi(\frac{z}{\sigma})$, with $\varphi:\bou(0_{\R^k},1)\to \R$ a compactly supported function satisfying $\int \varphi =1$.

Let us also set for $v\in \W^{1,p}(F\times\bou(0_{\R^k},\alpha))$
\begin{equation}\label{eq_tet}
\tilde{\Theta}v(x,y)=\Theta^Fv^y(x),
\end{equation}where $\Theta^F$ is the operator defined in section \ref{sect_theta}, for the manifold $F$ (we assume here that $U_0$ is sufficiently small for $\Theta^F$ to be defined). 

In order to generalize the argument we used in the case $k=0$, we first check that $\tilde{\Theta}$ satisfies a similar inequality as $\Theta^F$  (see (\ref{theta_contm})). Indeed, if we set $F^\eta:=F\cap \bou(0_{\R^n},\eta)$,  then, by (\ref{theta_contm}) for $\Theta^F$, we see that there is a constant $C$ (independent of $y$) such that for $v\in \W^{1,p}(F\times \bou(0_{\R^k},\alpha))$ (for all $p$ large enough)
\begin{equation*}
||(\tilde{\Theta} v)^y||^p_{L^p(F^\eta)}\le C  \eta^{p}|| v^y||^p_{\W^{1,p}(F^\eta)}.
\end{equation*}
Integrating with respect to $y$, we deduce  that
%Applying  (\ref{eq_coarea_sph}) (for the manifold $F$) and  %such that for almost every $y$
%\begin{equation}\label{eq_theta_y_F} ||(\tilde{\Theta} v)^y||_{L^p(F^\eta)} \le C \eta || v||_{\W^{1,p}(F^\eta)} .  \end{equation}
%This implies that 
\begin{equation}\label{eq_theta_y_U} ||\tilde{\Theta} v||_{L^p(E^\eta)} \lesssim \eta || v||_{\W^{1,p}(E^\eta)},   \end{equation}
where $E^\eta=F^\eta \times \bou(0_{\R^k},\alpha)$.

By Fubini's Theorem, we have for all $p$ large enough:
\begin{equation}\label{eq_commute}
\tilde{\Theta}\Phi_\sigma=\Phi_\sigma\tilde{\Theta}.
\end{equation}
We now claim that this entails that $\tilde{\Theta} \Phi_\sigma v$ belongs to $\W^{1,p}(U')$ for every $v$ in this space.  Indeed, thanks to Young's inequality, we see that (\ref{eq_theta_y_U}) entails that $\tilde{\Theta}\Phi_\sigma v$ is $L^p$. Moreover,  by (\ref{eq_commute}) $$\frac{\pa}{\pa y_i} \tilde{\Theta} \Phi_\sigma v(x,y)=\frac{\pa}{\pa y_i} \Phi_\sigma  \tilde{\Theta}v(x,y)=\left( (\tilde{\Theta}v)_x * \frac{\pa \varphi_\sigma}{\pa y_i}\right)(y),$$
which belongs to $L^p(U')$. Furthermore,  as $\tet \Phi_\sigma v(x,y)=\Theta^F (\Phi_\sigma v)^y(x)$, Lemma \ref{lem_cont} entails that the partial derivative of $\tet \Phi_\sigma v$ with respect to $x$ is also $L^p$, yielding the claimed fact.   

As $\Lambda$ is bi-Lipschitz, the function $u':=u\circ \Lambda^{-1}$ belongs to $\W^{1,p}(U')$. By (\ref{eq_commute}), we thus can write 
\begin{eqnarray}\label{eq_phi_sigma_uba}
\Phi_\sigma u'=\Phi_\sigma(u'-\tilde{\Theta}u')+\tilde{\Theta}\Phi_\sigma u'.
\end{eqnarray}
Let $f_\sigma$ and $g_\sigma$ respectively denote the first and second term composing the right-hand-side.

We first check that $f_\sigma \circ \Lambda$ is smooth. Since $M$ is normal, by Lemma \ref{lem_cont}, $f_\sigma (x,y)$ is constant with respect to $x$ and, as this function is smooth with respect to $y$, it induces a smooth function $\tilde{f}$ on $W_0$. As a matter of fact, since $\pi(\Lambda^{-1}(x,y))=y$, we see that $f_\sigma\circ \Lambda(z)= \tilde{f}\circ \pi(z)$ is smooth.

But since $\Phi_\sigma u'$ tends to $u'$ in $\W^{1,p} (U')$ (as $\sigma$ goes to zero), thanks to (\ref{eq_phi_sigma_uba}), it means that we just have to find arbitrarily close smooth approximations of  $g_\sigma\circ \Lambda$ in $\W^{1,p}(U_0)$ (for each $\sigma>0$ small), which, thanks to our induction hypothesis, reduces to checking that there are arbitrarily close approximations $v$ of $g_\sigma$ (in $\W^{1,p}(U')$) satisfying $\kappa_v>k$.

For this purpose, fix $\sigma>0$ and let for $\mu$ positive $\psit_\mu$  denote the ``cut-off function'' defined by $\psit_\mu(x,y)=\psi(\frac{|x|^2}{\mu^2})$, for $(x,y)\in U'$, where $\psi:\R\to [0,1]$ is (as in the case $k=0$) a $\Cc^\infty$ function  
such that $\psi\equiv 0$ on $(-\infty,\frac{1}{2})$ and $\psi\equiv 1$ near $[1,+\infty)$.  We claim that $\tilde{v}_\mu:=\psit_\mu \cdot g_\sigma $ is the desired sequence of approximations of $ g_\sigma $. 

Indeed, as $\tilde{v}_\mu$ tends to $g_\sigma$ in the $L^p$ norm, it actually suffices to establish the convergence of the first derivative. Remark that  since $\supp_{U'} |\pa \psit_\mu| \subset E^{\mu}$ and $\sup_{U'} |\pa \tilde{\psi}_\mu|\lesssim \frac{1}{\mu}$, we have (for all $p$ large enough)
$$||g_\sigma\cdot \pa \psit_\mu  ||_{L^p(U')}  \lesssim \frac{||g_\sigma  ||_{L^p(E^{\mu})} }{\mu}\overset{(\ref{eq_theta_y_U})}{\lesssim} ||\Phi_\sigma u'  ||_{\W^{1,p}(E^{\mu})},$$
which tends to zero as $\mu$ goes to zero. Hence,
\begin{equation}\label{eq_cv_vtilda}||\pa \tilde{v}_\mu -\pa g_\sigma||_{L^p(U')}\le  ||g_\sigma\cdot \pa \psit_\mu  ||_{L^p(U')} +||(1 -\psit_\mu) \cdot \pa g_\sigma||_{L^p(U')}\end{equation}                                                                                                                                                   
also tends to zero, yielding $(i)$.

\medskip

\noindent {\bf Proof of $(ii)$.} We  argue by induction on $m$ (the case $m=0$ follows from Lemma \ref{lem_cont}).
Fix a subanalytic set $A\subset \delta M $ of dimension $k$, and let $\Sigma$ be a stratification of $\adh{M}$ compatible with $A$ and $\delta M$.  
 Observe that if we show the result for every stratum $Y\subset A$ of dimension $k$, this will yield the result for $A$. Since we can work up to a partition of unity, we will work in the vicinity of a point of $\adh{Y}$, $Y\subset A$  stratum, that we will assume to be the origin. We set $Y^\eta:=\bou(0,\eta)\cap Y$, for $\eta\le\ep$.
 
 Let $u\in \Cc^0(\adh{M})\cap \W^{1,p}(M)$, with $\supp_{\adh{M}} u\subset \bou(0_{\R^n},\ep)$, where $\ep>0$ is sufficiently small for the operators $\Rr^M$ and $\Theta^M$ to be well-defined.  As $\Rr^M$ is a bounded operator and because $\Rr^M u$ is a constant function,  we clearly have  (for all $p$ large enough)$$||\Rr^M u||_{L^p(Y^{\ep})} \lesssim ||u||_{\W^{1,p}(M^\ep)}.$$ We thus simply have to focus on $\Theta^M u$, which means (since $\Theta^M$ is a projection) that we can assume that  $u=\Theta^M u$.   Notice first that we thus have for all $s\in (0,1)$ (and $p$ large):$$|| u\circ r_s||^p_{L^p (N^\ep)}\overset{(\ref{ln})}{\lesssim}
  s^{-\nu}||u||_{L^p(N^{s\ep}) } ^p \overset{(\ref{theta_cont})}{\lesssim}s^{p-1-\nu} ||u||_{\W^{1,p}(M^{s\ep})} ^p.$$
Moreover,  we have for all $s\in (0,1)$ 
$$||\pa (u\circ r_s)||_{L^p (N^\ep)}^p= ||^{\mathbf{t}} Dr_s (\pa u\circ r_s)||_{L^p (N^\ep)}^p\lesssim s^p ||\pa u\circ r_s||_{L^p (N^\ep)}^p      \overset{(\ref{ln})}{\lesssim}s^{p-\nu} ||\pa u||^p_{L^p (N^{s\ep})} .
$$
These two estimates give for $\eta<\ep$ and $p$ sufficiently large (using (\ref{eq_coarea_sph}))
\begin{equation}\label{eq_u_circ_r_s}
 \int_0 ^{\frac{\eta}{\ep}}||u\circ r_s||_{\W^{1,p} (N^\ep)}^p ds\lesssim \eta^{p-\nu}  ||u||^p_{\W^{1,p}(M^{\eta})}.
\end{equation}

 Given $\eta\le \ep$, we set for simplicity $Z^\eta:= \sph(0,\eta)\cap Y$. Observe that since $N^\ep$ has dimension $(m-1)$, the induction hypothesis implies that there is a constant $C$ such that for all $v \in \W^{1,p}(N^\ep)\cap \Cc^0(\adh{N^\ep})$  we have (for $p$ large enough)  \begin{equation}\label{eq_trace_rec}
||v||_{L^p(Z^\ep)}\le  C ||v||_{\W^{1,p}(N^\ep)}.                                                                                                                                                                                                                      \end{equation}
The mapping  $r_s$ can be required to preserve the strata (see Remark 2.6 of \cite{gvpoincare}), which now makes it possible to
 estimate the $L^p$ norm of the trace on $Y$ as follows:
$$||u||_{L^p(Y^\ep)}^p\overset{(\ref{eq_coarea_sph})}{\lesssim} \int_0 ^\ep ||u||^p_{L^p(Z^\zeta)}d\zeta \lesssim  \int_0 ^1 ||u\circ r_s||^p_{L^p(Z^{\ep})}ds \overset{(\ref{eq_trace_rec})}{\lesssim}\int_0^1 ||u\circ r_s||^p_{\W^{1,p}(N^\ep)} ds,  $$
which, by (\ref{eq_u_circ_r_s}) (for $\eta=\ep$), gives the desired bound for $||u||_{L^p(Y^\ep)}$. 

Remark that (this will be needed to prove $(iii)$) for $0<\eta<\ep$, integrating on $(0,\frac{\eta}{\ep})$ (instead of $(0,1)$),  the just above computation actually yields the more general estimate  (assuming as above $\Theta^M u=u$):\begin{equation}\label{eq_norme_trace_local}
      ||u||_{L^p(Y^\eta)}^p \lesssim  \int_0 ^\frac{\eta}{\ep} ||u\circ r_s||^p_{L^p(Z^{\ep})}ds  \overset{(\ref{eq_trace_rec})}{\lesssim} \int_0^{\frac{\eta}{\ep}} ||u\circ r_s||^p_{\W^{1,p}(N^\ep)} ds    \overset{(\ref{eq_u_circ_r_s})}{\lesssim}    \eta^{p-\nu}  ||u||^p_{\W^{1,p}(M^{\eta})}.                                                                                                                                                             \end{equation}

% comes down from (for $p>\nu+1$).
% since $r_s$ is Lipschitz. By , this entails that 
% $$||u||_{L^p(S)}^p \lesssim \int_0 ^\ep ||u||_{L^p(Z^\eta)}^p d\eta \lesssim \int_0 ^\ep s^{-\nu} ||u\circ r_s||_{L^p(Z^\ep)}^p ds  $$

\medskip

\noindent {\bf Proof of $(iii)$.} Fix  a  stratification $\mathcal{S}$ of $A\subset \delta M$. 
Take then a locally bi-Lipschitz trivial stratification $\Sigma$ of $\adh{M}$, compatible with   $\delta M$ as well as with all the strata of $\mathcal{S}$. Let $\Sigma_A$ be the stratification of $A$ induced by $\Sigma$ and denote by $\Sigma'_A$ the collection of all the elements of $\Sigma_A$ that are maximal for the partial order relation: $Y\preceq Y'$ if and only if $Y\subset \adh{Y'}$. Note that $\Sigma'_A$ is a stratification of a dense subset of $A$ which is open in $A$.

By definition of the trace, we have $\bigcap_{Y\in \mathcal{S}} \ker \tra_Y\subset \bigcap_{Y\in \Sigma'_A} \ker \tra_Y$.
 We are going to show that we can approximate every  function $u  \in \bigcap_{Y\in \Sigma_A'} \ker \tra_Y$ by decreasing induction on $\kappa_u$ (see (\ref{eq_kappa_u}) for $\kappa_u$). Take for this purpose such a function $u:M\to \R$ and suppose  the result true for every  $v\in \bigcap_{Y\in \Sigma'_ A} \ker \tra_Y$ satisfying $\kappa_v>\kappa_u$ (the case $\kappa_u=m$ is well-known).

% Using a partition of unity if necessary, we may assume that $\supp_{\adh{M}} u$ fits in some open set $U_{x_0}$, where $U_{x_0}$ is as in the proof of $(i)$. 
% Take $x_0\in \mba$ and let $S$ denote the element of $\Sigma$ that contains this point.  

As in the proof of $(i)$,  we will work near  $x_0=0_{\R^n}\in \delta M$, assuming $k:=\dim S=\kappa_u<m$, where $S$ is the stratum of $\Sigma$ containing $0_{\R^n}$.  We will also suppose that $\supp_{\mba}u\subset  U_{0}$, where $U_{0}$ is a neighborhood of $0_{\R^n}$ on which we have a bi-Lipschitz trivialization of $\Sigma$, and that $0_{\R^n}\in \adh{A}$  (since otherwise the desired fact follows from assertion $(i)$). Note  that there is a stratum $Y\in \Sigma_A'$ that contains $S$ in its closure.

We start with the easier case $\dim S=0$. %We shall find arbitrarily close approximations of $u$ in $\W^{1,p}(M)$  which have compact support in $\adh{M^\ep}\setminus 0$. 
Since $\tra_Y u=0$,  by definition of $\Rr^M$, we see that $\tra_Y \Rr^Mu=-\tra_Y \Theta^M u $, which  implies that  
\begin{equation*}
||\tra_Y \Rr^M u||_{L^p(Y^\eta)}^p= ||\tra_Y \Theta^M u||_{L^p(Y^\eta)}^p\overset{(\ref{eq_norme_trace_local})}{\lesssim}   \eta^{p-\nu}  ||\Theta^M u||^p_{\W^{1,p}(M^{\eta})}.
\end{equation*}
As $\Rr^M u$ is constant, for $p$ sufficiently large, this clearly entails that  $\Rr^M u$ is zero (for $p$ sufficiently large we have $\eta^{p-\nu} \lesssim \hn^l(Y^\eta)$, where $l=\dim Y$, see for instance \cite{lr} or \cite[Proposition 5.1.4]{livre}), which means that $u=\Theta^M u$. Consequently (see (\ref{eq_v_mu_cv})), $$v_\mu:=\psi_\mu \cdot \Theta^M u=\psi_\mu\cdot u$$ tends  to $u$ in $\W^{1,p}(M)$ as $\mu\in(0,1)$ tends to zero  (here $\psi_\mu$ is as in the proof of $(i)$). As $\kappa_{v_\mu}>0$ and $\tra_{Y'} v_\mu=\psi_\mu\cdot\tra_{Y'} u = 0$ for all $Y'\in \Sigma'_ A$ and all $\mu>0$ small, we see that the induction hypothesis ensures that we can find suitable approximations of every  $v_\mu$, which completes the induction step in the case $k=0$.

It remains the case $k>0$ that we are going to address with the same argument, using the operator $\tet$ that we constructed in the proof of $(i)$.  As $\Sigma$ is a locally bi-Lipschitz trivial stratification, there exists a bi-Lipschitz homeomorphism $\Lambda:U_{0}\to (\pi^{-1}(0_{\R^n})\cap \adh{M})\times \bou(0_{\R^k},\alpha)$, where $\pi$ is a smooth retraction onto $S$ and $\alpha>0$. %, and $U'= F\times \bou^k(x_0,\alpha)$, $\alpha>0$.  

For $v\in \W^{1,p}(F\times \bou(0_{\R^k},\alpha))$, set  $\tilde{\Rr}v(x,y):=\Rr^F v^y(x)$, where (as above) $F=\pi^{-1}(0_{\R^n})\cap M$, and remark that, if $\tet$ is as defined in (\ref{eq_tet}), we have: 
 \begin{equation}\label{eq_u_tet}v=\tet v +\tilde{\Rr} v.\end{equation}

 We claim that if we set $$u':=u\circ \Lambda^{-1}\et w_\sigma := \Phi_\sigma u'$$ (see (\ref{eq_Phi_sigma_def}) for $\Phi_\sigma$) %satisfies,where we have set
then $\tilde{\Rr} w_\sigma=0$ for all $\sigma>0$ small. 
 Indeed,   %\begin{equation}\label{eq_tr_w_sigma}\tra_Y w_\sigma=\Phi_\sigma \tra_Y u'=0. \end{equation} Notice that,
 by Lemma \ref{lem_cont}, $\tilde{\Rr}w_{\sigma}(x,y)$ is constant with respect to $x\in F$, and hence
 so is $\tra_Y \tilde{\Rr}w_{\sigma}(x,y)$.
 As a matter of fact, for all $\eta<\ep$ and all $x\in Y^\eta=\bou(0_{\R^n},\eta)\cap Y$, we have
 \begin{equation}\label{ntr1}
 ||\tra_Y\tilde{\Rr}w_\sigma||^p_{L^p(Y^\eta)}=\hn^{l-k}(Y^\eta\cap F)\cdot\int_{\bou(0_{\R^k},\alpha)}| \tilde{\Rr}w_\sigma|^p(x,y)\ dy,
 \end{equation} 
where $l=\dim Y$.   Moreover, it directly follows from (\ref{eq_norme_trace_local}) (for the manifold $F$) that  
                                      \begin{equation*}\label{eq_norme_trace_tet}
                                       ||\tra_Y\tet w_\sigma||_{L^p (Y^\eta)}^p \lesssim \eta^{p-\nu} ||\tet w_\sigma ||^p_{\W^{1,p} (E^\eta)} ,
                                      \end{equation*}
where we recall that $E^\eta= (F\cap \bou(0,\eta))\times \bou(0_{\R^k},\alpha)$.
As $\tra_Y w_\sigma=\Phi_\sigma \tra_Y u'=0$, thanks to  (\ref{eq_u_tet}), we deduce
% As a matter of fact, by (\ref{eq_u_tet}), we must have
 \begin{equation}\label{ntr2}
 ||\tra_Y \tilde{\Rr} w_\sigma||^p_{L^p(Y^\eta)}\lesssim \eta^{p-\nu}  ||\tet w_\sigma||^p_{\W^{1,p}(E^{\eta})}
\end{equation}
From (\ref{ntr1}) and (\ref{ntr2}), we conclude that $\tilde{\Rr}w_\sigma=0$ (for sufficiently large $p$, we have $ \eta^{p-\nu} \lesssim \hn^{l-k}(Y^\eta\cap F)$, again see \cite{lr} or \cite[Proposition 5.1.4]{livre}), as claimed.

Hence, $w_\sigma=\tet w_\sigma$, and 
consequently (see (\ref{eq_cv_vtilda})), for each $\sigma$,   $$\tilde{v}_\mu:=\psit_\mu\cdot w_\sigma=\psit_\mu \cdot \tet w_\sigma=\psit_\mu\cdot \tet \Phi_\sigma u'$$
 tends to $w_\sigma$ in $\W^{1,p}(M)$ as $\mu$ tends to zero. Since $\kappa_{\tilde{v}_\mu}> k$  and $\tra_{Y'} \tilde{v}_\mu=\psit_\mu\cdot\tra_{Y'} w_\sigma = 0$ for all $Y'\in \Sigma'_ A$, and as $w_\sigma$ tends to $u'$ as $\sigma$ goes to $0$, this completes the induction step.
\end{proof}%\lesssim \int_0 ^1 ||u||^p_{L^p(S^{s\ep})}ds %\overset{(\ref{eq_u_circ_r_s})}{\lesssim}  ||u||_{\W^{1,p}(M^{s\ep})}=
\begin{rem}\label{rem_addendum} \begin{enumerate}[(1)]
                                 \item \label{item_cont} The proof of $(ii)$ has actually established that, alike in the case of Lipschitz manifolds with boundary, if $u$ belongs to $\Cc^0(\adh{M})\cap \W^{1,p}(M)$ (and not necessarily to $\Cc^\infty(\adh{M})$) then $\tra_A u=u_{|A}$. 
                                \item\label{item_stra}  For  $(iii)$ of the above theorem to be true, it is not really necessary to require the manifold to be normal.  We have proved that it suffices that $M$ be connected at every point of  $\delta M\setminus \adh{A}$. It is as well worthy of notice that we do not need $\mathcal{S}$ to be a stratification of $A$, as it suffices to have a covering a dense subset of $A$ by some subanalytic sets $Y_0,\dots, Y_l$, $l=\dim A$, where for each $i$, $Y_i$ is either of pure dimension $i$ or empty.
                                \end{enumerate}

\end{rem}
\end{subsection}
\end{section}

\begin{section}{The general case}
\begin{subsection}{Normalizations} We introduce the  (subanalytic) $\Cc^\infty$ normalizations and show their existence and uniqueness. This will be needed to investigate the trace operator on non normal manifolds. This notion of $\Cc^\infty$ normalization is a natural counterpart of topological normalizations of pseudomanifolds \cite{mcc,ih1} in our framework.  

\begin{dfn}\label{dfn_normalization}
{\bf A $\Cc^\infty$ normalization of $M$ }
 is  a definable $\Cc^\infty$ diffeomorphism $h: \mc\to M$ satisfying $\sup_{x\in \mc} |D_x h|<\infty$ and  $\sup_{x\in M} |D_x h^{-1}|<\infty$,  with $\mc$ a normal $\Cc^\infty$ submanifold of $\R^k$, for some $k$. %such that $|d_xh|$ and $|d_xh^{-1}|$ are bounded on $M$ away from infinity.
 \end{dfn}

 So far, we have regarded the manifold $M$ as endowed with the metric induced by the euclidean metric of $\R^n$. We can also regard it as a Riemannian manifold:   given $x$ and  $y$ in the same connected component of $M$, let $d_M(x,y)$ denote the length of the shortest $\Cc^1$ arc in $M$ joining $x$ and $y$. 
 
 This metric is sometimes referred as {\bf the inner metric} of $M$.
 Although it is  not always equivalent to the euclidean metric, 
it enjoys the following property that will be useful for our purpose. Let $(A_t)_{t\in \R^k}$ be a definable family of manifolds (in the sense that $\bigcup_{t\in \R^k} A_t\times \{t\}$ is a definable set and $A_t$ is a smooth manifold for every $t$).  There is a constant $C$ (independent of $t$)  such that for all $t\in \R^k$ we have for all $x$ and $y$ in the same connected component of $A_t$:
 \begin{equation}\label{eq_diam_inner}
  d_{A_t}(x,y)\le C \mbox{diam}(A_t),
 \end{equation}
where $\mbox{diam}(A_t)$ stands for the euclidean diameter of $A_t$, defined as $\mbox{diam}(A_t)=\sup\{ |a-b|:a\in A_t, b\in A_t\}$. See \cite[Corollary 1.3]{kp} (or \cite[Proposition 3.2.6]{livre})  for a proof.

% If the manifold $M$ is not assumed to be normal it is still possible to define a multivalued trace by considering a normalization.

 \begin{lem}\label{lem_local_normal_extension}
   Suppose that $M$ is connected at $x_0\in \delta M$ and let $f:M\to \R$ be a definable $\Cc^1$ function. If $\sup_{x\in M} |\pa f(x)|<\infty$  then $f$ extends continuously at $x_0$.  
  \end{lem}
  \begin{proof} As $f$ has bounded derivative, by (\ref{eq_diam_inner}), it is bounded near $x_0$, and therefore  $\lim_{t \to 0} f(\gamma(t))$ exists for every definable arc $\gamma:(0,1)\to M$ tending to $x_0$ as $t$ goes to zero. It thus suffices to check that this limit is independent of the arc $\gamma$.
  
  Let for $\ep>0$, $A_\ep:= \sph(x_0,\ep)\cap M$ and observe that $A_\ep$ is connected for $\ep>0$ small.
By (\ref{eq_diam_inner}), there is a constant $C$ such that  for any $\ep>0$ small, $d_{A_\ep}(x,y)\le C\ep$, for all $x$ and $y$ in $A_\ep$. 
As $f$ has bounded derivative, this entails that $|f(x)-f(y)|\le C'\ep$, for all such $x$ and $y$, where $C'$ is a constant. Consequently, if $\gamma_1$ and $\gamma_2$ are two definable arcs in $M$  tending to $x_0$ (that we can parameterize by their distance to  $\xo$), $|f(\gamma_1(t))-f(\gamma_2(t))|$ tends to $0$ as $t$ goes to $0$.
   \end{proof}
 
Observe that if $h:\mc\to M$ is  a normalization and if $z\in \delta \mc$ then this lemma entails that $h$ induces a homeomorphism between the germ of $\mc$ at $z$ and a connected component of the germ of $M$ at $\lim_{x\to z} h(x)$.  It also yields:
%Consequently, for $\eta$ small enough and $z \in cl(\mc)$, the restriction of
% $h$ to $\bou(z,\eta)\cap \mc  $,   is a homeomorphism onto its image. The proof of the above proposition has established that for $x_0 \in cl(\mc)$ the restriction of $h^{-1}: M \to \mc$ to a connected component $C$ of $h(\bou(x_0;\eta) \cap \mc)$ extends continuously to a homeomorphism between $cl(C)$ and $cl(\bou(x_0,\eta)\cap cl(\mc))$.  faire une pro pour dire que ca s'etend a un homeo quand c'est ib-Lip wrt inner et connexe, et en
% An immediate consequence of the above lemma is the following result about normalizations.

  \begin{pro}\label{pro_normalization_uniqueness} Every $\Cc^\infty$ normalization  $h: \mc\to M$  extends continuously to a mapping from $\adh{\mc}$ to $\mba$.
    Moreover, if $h_1:\mc_1 \to M$ and $h_2:\mc_2 \to M$ are two $\Cc^\infty$ normalizations of $M$, then $h_2^{-1} h_1$  extends to a homeomorphism between $\adh{\mc_1}$ and $\adh{\mc_2}$.%, bi-Lipschitz with respect to the inner metric.  
  \end{pro}
%   \begin{proof}
%    
% Take two normalizations $\pi_1:cl(\mc_1)\to M$ and $\pi_2:\mc_2\to M$.  Then, the mapping $\pi:= \pi_2^{-1}\pi_{1|\mc_1}$ is a homeomorphism.  We are going to show that it extends to $cl(\mc_1)$. Take two analytic arcs $\alpha:]0,\ep[\to \mc_1$ and $\beta:]0,\ep[\to \mc_1$, tending to the same point $x_0\in X$ of $cl(\mc_1)$ as $t$ goes to zero and suppose that $h(\alpha)$ and $h(\beta)$ have different limits $a$ and $b$ (these limits may not be infinity since $h$ is proper). 
% 
% As $h$ is a homeomorphism and $\mc_1$ normal, $h(B(x_0,\eta) \cap \mc_1)$ is connected for $\eta$ small enough.  For $t$ small enough  $\alpha(t)$ and $\beta(t)$ are in this ball and can be joint by and arc $\gamma_t$ in this ball. Consider the family $A:=cl(\{(s,t,\gamma_t(s)) : 0<t\leq 1\})$ and let $A_0$ be the fiber at $t=0$. Since $A_0$ is connected, it is infinite. Since $\lim_{t \to 0} \gamma_t\equiv x_0$, the subset $  A_0$ is mapped by $h_2$ onto $h_1(x_0)$. This  contradicts the fact that $h_2$ is finite to one.
% 
% By symmetry of the roles of $\mc_1$ and $\mc_2$, $h^{-1}$ also extends continuously. 
% \end{proof}

For $x\in \delta M$, let \begin{equation*}\label{eq_c_m}
                                                                           \cfr_M(x):=\mbox{ number of connected components of $\bou(x,\ep)\cap M$, $\ep>0$ small.}
                                                                          \end{equation*} 
 If $\overline{h}$ stands for the extension of $h$ to $\adh{\mc}$, then $ \overline{h}^{-1}(x)$ has exactly $ \cfr_M(x)$ points, for each $x\in\delta M$. In particular, $\overline{h}$ must be finite-to-one. 
 Remark also that this proposition somehow establishes uniqueness of $\Cc^\infty$ normalizations, in the sense that the inner Lipschitz geometry of $\adh{\mc}$ is independent of the chosen normalization $\mc$. The next result is devoted to their existence.   

Given a submanifold $S\subset \R^n$ and $x\in \R^n$, we denote by $\pi_S(x)$  the point of $S$ at which the euclidean distance to $S$ is reached (this point is unique if $x$ is sufficiently close to $S$) and by $\rho_S(x)$ the square of this distance. These two mappings, defined on a suitable neighborhood $U_S$ of $S$,  are globally subanalytic and  $\Cc^\infty$ if $U_S$ is small enough (see for instance \cite{livre}). We then say that $(U_S,\pi_S,\rho_S)$ is a {\bf tubular neighborhood of $S$}.

 \begin{pro}\label{pro_normal_existence}
Every bounded definable manifold admits a $\Cc^\infty$ normalization.
\end{pro}
\begin{proof}
   Let $B_M$ be the set of points of $\delta M$  at which $\cfr_M(x)>1$ and set $k_M:=\dim B_M$. We are going to prove the result by induction on $k_M$ (if $k_M=-1=\dim \emptyset$, we are done). 
 
Take a stratification $\Sigma$ of $\delta M$  such that $\cfr_M$ is constant on the strata, and let $S$ be a $k_M$-dimensional stratum included in $B_M$. 
As the construction that we are going to carry
out may be realized simultaneously for every such $S$, we
will assume that $S$ is the only element of dimension $k_M$ of $\Sigma$ included in $B_M$.

Let $(U_S,\pi_S,\rho_S)$ be a tubular neighborhood of $S$ and denote by  $C_1,\dots,C_l$  the connected components of $M \cap
U_S$.  Taking $U_S$ smaller if necessary, we can assume that 
 $l=\cfr_M(x)$, for all $x\in S$.

Let $\alpha$ be a definable $\Cc^\infty$ function on $U_S$ satisfying $|\alpha(x)-d(x,\R^n \setminus U_S)|<\frac{d(x,\R^n \setminus U_S)}{2}$ (by Efroymson's approximation theorem for globally subanalytic functions \cite{ef}, such a function exists). It follows from \L ojasiewicz's inequality that there is an integer $\kappa_1 \ge 1$ such that $\alpha^{\kappa_1-1} \lesssim\frac{1}{1+|\pa \alpha|}$ on $U_S$, so that the function $\alpha^{\kappa_1}$ has bounded gradient on this set. Similarly, there is a definable $\Cc^\infty$ function $\beta$ on $\R^n\setminus\delta S$ such that $|\beta(x)-d(x,\delta S)|<\frac{d(x,\delta S)}{2}$ and an integer $\kappa_2$ such that  $\beta^{\kappa_2}$ has bounded gradient. We claim that, choosing $\kappa_2$ bigger if necessary, we can require in addition that there is a constant $A$ such that 
on the set $$Z:=\{x\in U_S: \rho_S(x) \le \alpha(x) \},$$ we have:\begin{equation}\label{eq_alpha_beta}
\beta(x)^{\kappa_2} \le A \alpha(x)^{\kappa_1}. 
\end{equation}
Indeed, thanks to \L ojasiewicz's inequality, it suffices to check that $\lim_{t \to 0 }\beta(\gamma(t))=0$ for every definable arc $\gamma:[0,1] \to Z$ such that $\lim_{t\to 0} \alpha(\gamma(t))=0$. To see this,  note that such an arc $\gamma$ must tend to a point of $\delta U_S$ (since $\alpha$ tends to zero) which is a point of $\adh{S}$ (since $\rho_S$ also tends to zero, due to the definition of $Z$), which entails that $\lim_{t\to 0} \gamma(t) \in \delta S$ (since $\delta U_S\cap \adh S\subset \delta S$), from which we can conclude that $\beta$ tends to zero, as required. 

We now are going to construct for every $i\le l$ a definable $\Cc^\infty$ function $\rho_i$ on $M$  satisfying:
\begin{enumerate}[(a)]
\item There is a neighborhood $V_S$ of $S$ such that $\rho_i>\frac{\beta^{\kappa_2}}{2}$ on $C_i\cap V_S$, and $\rho_i<\frac{\beta^{\kappa_2}}{3}$ on $C_j\cap V_S$ for every $j\ne i$.  
\item $\rho_i$ has bounded first derivative and tends to zero as we are drawing near $\delta S$. %For all $j$, the restriction of $\rho_i$ to $C_j$ extends  continuously  on $cl(C_j)$.
\end{enumerate}

Choose a  $\Cc^1$ definable function $\phi:\R \to \R$ such that $\phi(t)=0$ for $t\ge 1$, $\phi(t)=1$ for $t<\frac{1}{2}$ and let $i\le l$.  Define then a $\Cc^1$ function on $C_i$ by setting for $x\in C_i$
$$\rho_i(x):= \phi\left(\frac{\rho_S(x)^{\kappa_1}}{\alpha(x)^{\kappa_1}}\right)\cdot \beta(x)^{\kappa_2},$$
and extend this function as identically zero on  $M\setminus C_i$.  We shall establish that $\rho_i$ satisfies properties $(a)$ and $(b)$. This function $\rho_i$ is not $\Cc^\infty$ but just $\Cc^1$. However,
as properties $(a)$ and $(b)$ also hold for any sufficiently close approximation of $\rho_i$  (in the $\Cc^1$ topology, see \cite{ef}),  the existence of a $\Cc^\infty$ function satisfying these properties will then follow from Efroymson's approximation theorem for globally subanalytic functions \cite{ef}. 

%We are going to prove that a (sufficiently close) $C^\infty$ approximation  $\rho_i$  of $\rho_i$ has the required properties. Indeed 
%Note that since $ \phi(\frac{\rho_S(x)}{d(x,\R^n\setminus U_S)})$ is equal to $1$ near $S$ and $\mu'$ extends continuously on $S$,  $\rho_i$ extends continuously on $S$.
Thanks to (\ref{eq_alpha_beta}),   a straightforward computation of derivative yields that $\rho_i$ has bounded first derivative (note that the support of the function $ \phi'\left(\frac{\rho_S(x)^{\kappa_1}}{\alpha(x)^{\kappa_1}}\right)$ is included in the above set $Z$). Moreover, 
as $\beta$ tends to zero as $x$ tends to $\delta S$, it is clear that so does $\rho_i$, which yields $(b)$. Fact $(a)$ also holds since $\rho_i$ vanishes on $C_j$ for $j\ne i$, and $\phi\equiv  1$ near $S$.   
  %$\alpha$ on $\R^n$
%zero on the complement of $U_S$ and which does not vanish on $U_S$.
%We may approximate this function by a function on $\R^n$, $C^2$ on $U_S$, still having the latter properties. Rising it to a huge power we get a   $C^2$ function  on $\R^n$ (thanks to \L
%ojasiewicz inequality).
% We can extend $\rho_{\xb}$ to a $C^1$ subanalytic function on an open  neighborhood $W$ of $\xb$.
%   Let $\rho_i$ be the restriction of this function to a neighborhood of $C_i$ in $M$. Extend this function to the whole of $M$ by $\rho_i(x):=0$ if $x \notin C_i$.

Let now for $x \in M$,  $\rho(x):=(\rho_1(x),\dots,\rho_l(x))\in \R^{l},$
and let $\check{M}$ stand for the graph of $\rho$. Denote by $\pi:\check{M}\to M$ the map induced by the canonical projection onto $\R^n$ and  by $\overline{\pi}:\adh{\check{M}}\to \mba$ its extension. As $\pa \rho_i$ is bounded for all $i$, the mappings $\pi$ and $\pi^{-1}$ both have bounded first derivative.

By $(a)$, the sets $\adh{\Gamma_{\rho_{|C_j}}}\cap \overline{\pi}^{-1}(S)$, $j\le l$, are disjoint from each other. Hence, 
  as for each $j$, the function $\rho_{|C_j}$ extends continuously,    $\adh{\Gamma_{\rho}}$ must be
 connected at every point of $\overline{\pi}^{-1}(S)$.   Observe also that, due to $(b)$  and Lemma \ref{lem_local_normal_extension}, at a given point $x\in \delta M$, the function $\rho$ cannot have more than $\cfr_M(x)$ asymptotic values, which means that $\overline{\pi}$ is finite-to-one and that $\mc$ is connected at every point above $\mba\setminus B_M$.  Hence, $\{x \in \mc :\cfr_{\mc}(x)>1\} \subset \overline{\pi}^{-1}(B_M\setminus S)$ (since $\cfr_{\mc}(x)\equiv 1$  on  $ \overline{\pi}^{-1}(S)\cup  \overline{\pi}^{-1}(\mba\setminus B_M) $), which has dimension less than $k_M$ (since  $\overline{\pi}$ is finite-to-one).  
 This completes the induction step.
  \end{proof}

  \begin{rem}\label{rem_normalisations_C_i} We have used the subanalytic version of  Efroymson's approximation theorem  \cite{ef} because we wanted to construct 
 $\Cc^\infty$ normalizations. If one only wishes to construct  $\Cc^i$ normalizations $i\in \N$ (defined analogously, they seem to be satisfying for many purposes), it is easy to spare this theorem which is rather difficult to prove. 
 \end{rem}
 \end{subsection}

\begin{subsection}{The general case.} If the manifold $M$ is not assumed to be normal, it is possible to define a multivalued trace by considering a normalization, and then prove some partial  generalizations of the results of section \ref{sect_case_normal}. We first point out some facts useful for this purpose. 

Fix a normalization $h:\mc\to M$.  
 As $h$ and $h^{-1}$ have bounded derivative the mapping $h_*:\W^{1,p}(\mc)\to \W^{1,p}(M)$, $u\mapsto u\circ h^{-1}$ is a continuous isomorphism for all $p$. Theorem \ref{thm_trace} thus immediately yields that  
 for $p\in [1,\infty)$ sufficiently large, the space \begin{equation}\label{eq_c_h} \Cc^h (\mba):=h_* \Cc^\infty (\adh{\mc})=\{u\circ h^{-1} : u \in \Cc^\infty(\adh{\mc}) \} \end{equation}                                                                                                                                                                            
is  dense in $\W^{1,p}(M)$.
% \begin{cor}\label{cor_h_star}
%  \end{cor}
% Clearly,   is independent of the chosen $\Cc^\infty$ normalization $h$ follows from Proposition \ref{pro_normalization_uniqueness} that.

% \begin{rem} \end{rem}
Although the functions of $ \Cc^h(\mba)$  may fail to be smooth on $\mba$, this ring is satisfying for many purposes. A given function $v$ of this ring  has the property that for every $\xo$ in $\mba$, the restriction of $v$ to a connected component $U$ of $\bou(\xo,\ep)\cap M$, $\ep>0$ small, extends to a function which is Lipschitz with respect to the inner metric. It actually can be required (see \cite{gvpoincare, livre}) that the gradient of this restriction  extends to a stratified $1$-form and therefore can be used to perform integrations by parts.

   By Proposition \ref{pro_normalization_uniqueness}, $h$ extends to a continuous mapping $\adh{h}:\adh{\mc} \to \adh{M}$. We also have:
\begin{pro}\label{pro_stra}
 There are stratifications $\check{\mathcal{S}}$ and $\mathcal{S}$ of $\delta \mc$ and $\delta M$ respectively such that for each $S\in \mathcal{S}$, $\adh{h}^{-1}(S)=\bigcup_{i=1}^j S_i$, where, for each $i\le j$, $S_i$ is a stratum of $\check{\mathcal{S}}$ on which $\adh{h}$ induces a diffeomorphism $h_{S_i}:S_i\to S$ satisfying $\sup_{x\in S_i} |D_x h_{S_i}|<\infty$ and $\sup_{x\in S} |D_x h_{S_i}^{-1}|<\infty$.   
\end{pro}
\begin{proof}Possibly replacing $\mc$ with the set $\{(y,x):y=h(x)\}$, we can assume that $h$ is given by a canonical projection. We then take a cell decomposition $\D$ compatible with $\adh{\mc}$ and $\mc$, which immediately gives rise to a cell decomposition $\E$ of $\R^n$ (see  \cite{livre} Remark $1.2.2$).  As $h$ is finite-to-one,  it maps the cells of $\St$ onto the cells of $\E$. These two cell decompositions induce stratifications of $\adh{\mc}$ and $\adh{M}$ respectively. That we can refine these stratifications to stratifications such that $h$ and $h^{-1}$ have bounded first derivative on the strata now follows from \cite{gvpoincare} Lemma $2.8$ (see also \cite{livre} Proposition $2.7.13$).
\end{proof}
\subsection*{Definition of the trace.} We first define $\tra_S$  when $S$ is a stratum of $\mathcal{S}$, where $\mathcal{S}$ is provided by Proposition \ref{pro_stra} above.  Let $l:=\sup_{x\in M} \cfr_M(x)$, fix $S\in \mathcal{S}$, and remark that $\cfr_M$ is constant on $S$. We thus can
  define $\tra_S : \W^{1,p}(M) \to L^p (S)^l$ by setting for $v\in \W^{1,p}(M)$  (and $p\in [1,\infty)$ large): $$\tra_S v := \left( (\tra_{S_1} v\circ h)\circ h_{S_1}^{-1}, \dots , (\tra_{S_j} v\circ h)\circ h_{S_j}^{-1}, 0,\dots, 0\right).$$  Here, we add $(l-j)$ times the zero function because  it will be convenient that the trace has the same number of components for all $S \in \mathcal{S}$. This mapping of course depends on the way the elements $S_1,\dots, S_j$ are enumerated (see nevertheless on this issue Proposition \ref{pro_trace} below). Notice that since $h_{S_i}:S_i\to S$ and its inverse have bounded derivative, Theorem \ref{thm_trace} ensures that $\tra_S$ is a bounded  operator.
  
  %We however have:

 Take now  a subanalytic subset $A$  of $\delta M$ and assume that $\mathcal{S}$ is compatible with $A$.   In this situation, we  can define $\tra_A v$ as the function induced by the mappings $\tra_S v$, $S\in \mathcal{S}$, $S\subset A$, $\dim S=\dim A$. Since the strata of dimension less than $k:=\dim A$ are $\hn^k$-negligible this definition is clearly independent of the chosen stratification. 
 
%  \subsection{Trace theorem and density results for  subanalytic manifolds.} 
%  
% \subsection*{Some more corollaries.}
Theorem \ref{thm_trace} now has the following consequence:
\begin{cor}\label{cor_trace_bounded}
Let $A\subset \delta M$ be a subanalytic set of dimension $k$.  For $p\in [1,\infty)$ sufficiently large, the linear operator $$\tra_A :\W^{1,p}(M) \to L^p(A,\hn^k)^l,$$ is bounded. 
\end{cor} 
% \begin{proof}
%  
% \end{proof}

%This corollary, although technical has some explicit consequences. If $S$ and $\mathcal{S}$ are as in Proposition \ref{pro_stra}
 We then denote by $\tra_{A,1},\dots \tra_{A,l}$ the components of $\tra_A$.  As pointed out before Corollary \ref{cor_trace_bounded}, the trace depends on the way the strata $S_1,\dots, S_j$ above are enumerated. This choice being made, we however have:
\begin{pro}\label{pro_trace}Let $A\subset \delta M$ be subanalytic. If $p$ is sufficiently large then, for every  $v\in \W^{1,p}(M)$,  
the set of functions $\{ \tra_{A,1}\, v , \dots , \tra_{A,l} \,v\}$ does not depend on the chosen $\Cc^\infty$ normalization. 
\end{pro}
\begin{proof}
 Let $h_1:\mc_1 \to M$ and  $h_2:\mc_2\to M$ be two $\Cc^\infty$ normalizations of $M$. By Proposition \ref{pro_normalization_uniqueness},  the mapping $H:=h_2^{-1}h_1:\mc_1 \to \mc_2$ is a diffeomorphism which extends continuously on $\delta \mc_1$, giving rise  to a homeomorphism  $\adh{H} :\adh{\mc_1}\to \adh{\mc_2}$. The result is therefore clear if $v\circ h^{-1}_1$ extends continuously to $\adh{\mc_1}$ (see Remark \ref{rem_addendum} (\ref{item_cont})),  which means that it holds for every  $v\in \Cc^{h_1}(M) $ (see (\ref{eq_c_h})).  But, as this set is dense and $\tra_A$ is bounded, it means that the result must hold for each $v\in \W^{1,p}(M)$.
%  As the mappings $H$ and $H^{-1}$ both have bounded first derivative, the mapping $H_*: \W^{1,p}(\mc)\to \W^{1,p}(\mc')$, $u\mapsto u\circ H^{-1}$ is a continuous isomorphism.  that, due to Proposition \ref{pro_stra}, must have bounded derivative on the strata of a stratification of $\delta \mc$ Of course, it suffices to prove the result in the case where $A$ is a stratum of a stratification of $\delta M$. 
\end{proof}

It is worthy of notice that Theorem \ref{thm_trace} also yields that if $\Sigma$ is any stratification of $A \subset \delta M$, then (for $p$ large) $h_* \Cc^\infty _{h^{-1}(\overline{M} \setminus \overline{A})}(\overline{\mc})$ is dense in $
\underset{S\in \Sigma}{\bigcap}\ker \tra_S $, where $h$ is any $\Cc^\infty$ normalization of $M$.  In particular, in the case where $A$ is dense in $\delta M$, we get:
\begin{cor}\label{cor_densite_nonnormal} If $\Sigma$ is a stratification of a dense subset of $\delta M$ then
 $ \Cc^\infty _0(M)$ is dense in $
\underset{S\in \Sigma}{\bigcap}\ker \tra_S $ for all $p\in [1,\infty)$ sufficiently large. 
\end{cor}

We also can get a converse of $(i)$ of Theorem \ref{thm_trace}:
%  The following result gives another consequence of Corollary \ref{cor_trace_bounded}.

\begin{cor}\label{cor_normal}
$\Cc^\infty(\adh{M})$ is dense in $\W^{1,p}(M)$ for arbitrarily large values of $p$ if and only if $M$ is normal.
\end{cor}
\begin{proof}The ``if'' part is established by Theorem \ref{thm_trace}. To prove the ``only if'' part, 
assume that $M$ fails to be normal and take a point $x_0\in\adh{M}$ at which  $M$ is not connected. For $\ep>0$ small enough, let $U_1,\dots, U_l$, $l>1$, be the connected components of $M\cap\bou(x_0,\ep)$. Let then $u:\bigcup_{j=1}^l U_j\to\R$ be the function identically equal to 1 on $U_1$ and $0$ on the other $U_j$'s. Multiplying this function by a $\Cc^\infty$ function $v$ which is $1$ on some neighborhood of $x_0$ in $\R^n$ and $0$ outside $\bou(x_0,\ep)$, we get a function $\tilde{v}$ which extends continuously to $M$. Clearly, this extension, still denoted $\tilde{v}$, belongs to $ \W^{1,p}(M)$.

Suppose now that there exists a sequence $(v_i)_{i \in \N} $ in $\Cc^\infty(\adh{M})$ tending to $\tilde{v}$ in $ \W^{1,p}(M)$. We have $\tra_{\{x_0\}}v_{i}=(v_{i}(x_0),\dots,v_i (\xo),0,\dots, 0)$, where $v_i(x_0)$ occurs $l$ times, while $\tra_{\{x_0\}}\tilde{v}= (1,0,\dots,0)$. Since, due to Corollary \ref{cor_trace_bounded}, $\tra_{\{x_0\}}v_{i}$ should tend to $\tra_{\{x_0\}}\tilde{v}$, this contradicts $l>1$.
\end{proof}

\begin{rem} Of course, if the manifold $M$ is unbounded, the trace is well-defined and $L^p_{loc}$ on the boundary.   Moreover, since we can use cutoff functions, the density results remain true in this case.
 \end{rem}
 \end{subsection}
 \end{section}

\end{document}